\documentclass[reqno]{amsart}

\usepackage{url}
\usepackage{a4wide}

\usepackage[utf8]{inputenc}
\usepackage[T1]{fontenc}
\usepackage{lmodern}

\usepackage{amsmath}
\usepackage{amssymb}
\usepackage{amsfonts}
\usepackage{amscd}
\usepackage{amsthm}
\usepackage{caption}
\usepackage{subcaption}

\usepackage{graphicx}
\usepackage[english]{babel}

\usepackage{wrapfig}
\usepackage{breqn}

\usepackage[all]{xy}
\usepackage{hyperref}

\usepackage[ruled,vlined,boxed,linesnumbered]{algorithm2e}

\usepackage{cleveref}
\makeatletter
\newcommand\footnoteref[1]{\protected@xdef\@thefnmark{\ref{#1}}\@footnotemark}
\makeatother

\allowdisplaybreaks

\theoremstyle{plain}
\newtheorem{theorem}{Theorem}[section]

\newtheorem{proposition}[theorem]{Proposition}
\newtheorem{lemma}[theorem]{Lemma}

\theoremstyle{definition}
\newtheorem{definition}[theorem]{Definition}

\theoremstyle{remark}
\newtheorem{remark}[theorem]{Remark}

\numberwithin{equation}{section}

\DeclareMathOperator{\Aut}{Aut}

\DeclareMathOperator{\Gal}{Gal}
\DeclareMathOperator{\GL}{GL}

\DeclareMathOperator{\Jac}{Jac}

\DeclareMathOperator{\PGL}{PGL}
\DeclareMathOperator{\PSL}{PSL}

\DeclareMathOperator{\SL}{SL}

\DeclareMathOperator{\Sp}{Sp}

\DeclareMathOperator{\Twist}{Twist}
\DeclareMathOperator{\Prob}{P}

\def\kbar{\overline{k}}

\def\ie{\textit{i.e. }}

\newcommand{\Id}{\textrm{id}}

\newcommand{\Mm}{\mathsf{M}}

\newcommand{\FF}{\mathbb{F}}
\newcommand{\CC}{\mathbb{C}}

\newcommand{\PP}{\mathbb{P}}

\newcommand{\RR}{\mathbb{R}}

\newcommand{\bfM}{\mathbf{M}}

\newcommand{\AG}{\mathbf{A}}
\newcommand{\CG}{\mathbf{C}}
\newcommand{\DG}{\mathbf{D}}
\newcommand{\Gg}{\mathbf{G}}

\newcommand{\SG}{\mathbf{S}}

\newcommand{\cc}{\mathcal{C}}

\newcommand{\sS}{\mathcal{S}}

\usepackage{mdwlist}

\hypersetup{
  pdfauthor   = {Reynald Lercier, Christophe Ritzenthaler, Florent Rovetta, Jeroen Sijsling}
  pdftitle    = {Parametrizing the moduli space of curves and applications to smooth plane quartics over finite fields}
  pdfsubject  = {14Q05 (primary), 13A50, 14H10, 14H37 (secondary)},
  pdfkeywords = {Genus 3 curves ; plane quartics ; moduli ; families ; enumeration ; finite fields},
  backref=true, pagebackref=true, hyperindex=true, colorlinks=true, breaklinks=true, urlcolor=black, linkcolor=black, citecolor=black, bookmarks=true, bookmarksopen=true}

\title[Parametrizing the moduli space]{Parametrizing the moduli space of curves
  and applications to smooth plane quartics over finite fields}

\author[Lercier]{Reynald Lercier}
\address{%
  Reynald Lercier, %
  DGA \& Univ Rennes, %
  CNRS, IRMAR - UMR 6625, F-35000
  Rennes, %
  France. %
}
\email{reynald.lercier@m4x.org}

\author[Ritzenthaler]{Christophe Ritzenthaler}
\address{%
  Christophe Ritzenthaler, %
  Univ Rennes, %
  CNRS, IRMAR - UMR 6625, %
  F-35000 Rennes, %
  France. %
}

\email{christophe.ritzenthaler@univ-rennes1.fr }

\author[Rovetta]{Florent Rovetta}
\address{Florent Rovetta, %
  Institut de Math{\'e}matiques de Luminy, %
  UMR 6206 du CNRS, %
  Luminy, Case 907, 13288 Marseille, France.}
\email{florent.rovetta@univ-amu.fr}

\author[Sijsling]{Jeroen Sijsling}
\address{Jeroen Sijsling, Mathematics Institute, Zeeman Building, University of
  Warwick, Coventry CV4 7AL, United Kingdom. }
\curraddr{Jeroen Sijsling, Institut für Reine Mathematik, Universität Ulm,
  Helmholtzstrasse 18, 89081 Ulm, Germany.}
\email{jeroen.sijsling@uni-ulm.de}

\thanks{The authors acknowledge support by grant ANR-09-BLAN-0020-01. The
fourth author was additionally supported by a Marie Curie Fellowship
IEF-GA-2011-299887.}

\date{\today}

\begin{document}

\begin{abstract}
  We study new families of curves that are suitable for efficiently
  parametrizing their moduli spaces. We explicitly construct such families for
  smooth plane quartics in order to determine unique representatives for the
  isomorphism classes of smooth plane quartics over finite fields. In this way,
  we can visualize  the distributions of their traces of Frobenius. This leads
  to new observations on fluctuations with respect to the limiting symmetry
  imposed by the theory of Katz and Sarnak.
\end{abstract}

\maketitle

\section{Introduction}
\label{sec:introduction}

One of the central notions in arithmetic geometry is the \emph{(coarse) moduli
space of curves}  of a given genus $g$, denoted $\Mm_g$. These are algebraic
varieties whose geometric points classify these curves up to isomorphism. The
main difficulty when dealing with moduli spaces -- without extra structure --
is the non-existence of \emph{universal families}, whose construction would
allow one to explicitly write down the curve corresponding to a point of this
space. Over finite fields, the existence of a universal family would lead to
optimal algorithms to write down isomorphism classes of curves. Having these
classes at one's disposal is useful in many applications. For instance, it
serves for constructing curves with many points using class field
theory~\cite{rokaeus} or for enlarging the set of curves useful for
pairing-based cryptography as illustrated in genus $2$
by~\cite{FS11,gv12,satoh}.  More theoretically, it was used in~\cite{berg} to
compute the cohomology of moduli spaces. We were ourselves drawn to this
subject by the study of Serre's obstruction for smooth plane quartics (see
Section~\ref{subsec:dist}).

The purpose of this paper is to introduce three substitutes for the notion of a
universal family. The best replacement for a universal family seems to be that
of a \emph{representative family}, which we define in
Section~\ref{sec:famdefs}. This is a family of curves $\cc \to \sS$ whose
points are in natural bijection with those of a given subvariety $S$ of the
moduli space. Often the scheme $\sS$ turns out to be isomorphic to $S$, but the
notion is flexible enough to still give worthwhile results when this is not the
case. Another interesting feature of these families is that they can be made
explicit in many cases when $S$ is a stratum of curves with a given
automorphism group. We focus here on the case of non-hyperelliptic genus $3$
curves, canonically realized as \emph{smooth plane quartics}.

The overview of this paper is as follows. In Section~\ref{sec:famdefs} we
introduce and study three new notions of families of curves. We indicate the
connections with known constructions from the literature. In
Proposition~\ref{prop:desrep} and Proposition~\ref{prop:iter}, we also uncover
a link between the existence of a representative family and the question of
whether the field of moduli of a curve is a field of definition. In
Section~\ref{sec:famquart} we restrict our considerations to the moduli space
of smooth plane quartics. After a review of the stratification of this moduli
space by automorphism groups, our main result in this section is
Theorem~\ref{thm:smallstrata}. There we construct representative families for
all but the two largest of these strata by applying the technique of Galois
descent.  For the remaining strata we improve on the results in the literature
by constructing families with fewer parameters, but here much room for
improvement remains. In particular, it would be nice to see an explicit
representative (and in this case universal) family over the stratum of smooth
plane quartics with trivial automorphism group.

Parametrizing by using our families, we get one representative curve per
$\bar{k}$-isomorphism class.  Section~\ref{sec:computation-twists} refines
these into $k$-isomorphism classes by constructing the \emph{twists} of the
corresponding curves over finite fields $k$.  Finally,
Section~\ref{sec:impl-exper} concludes the paper by describing the
implementation of our enumeration of smooth plane quartics over finite fields,
along with the experimental results obtained on distributions of traces of
Frobenius for these curves over $\FF_p$ with $11 \leq p \leq 53$. In order to
obtain exactly one representative for every isomorphism class of curves, we use
the previous results combined with an iterative strategy that constructs a
complete database of such representatives by ascending up the automorphism
strata\footnote{Databases and statistics summarizing our results can be found
  at~\url{http://perso.univ-rennes1.fr/christophe.ritzenthaler/programme/qdbstats-v3_0.tgz}.
}.\\

\textit{Notations.} Throughout, we denote by $k$ an arbitrary field of
characteristic $p \geq 0$, with algebraic closure $\kbar$. We use $K$ to denote
a general algebraically closed field. By $\zeta_n$, we denote a fixed choice of
$n$-th root of unity in $\kbar$ or $K$; these roots are chosen in such a way to
respect the standard compatibility conditions when raising to powers.  Given
$k$, a \emph{curve} over $k$ will be a smooth and proper absolutely irreducible
variety of dimension $1$ and genus $g>1$ over $k$.

In agreement with~\cite{lrs}, we keep the notation $\CG_n$ (resp. $\DG_{2n}$,
resp. $\AG_n$, resp. $\SG_n$) for the cyclic group of order $n$ (resp.\ the
dihedral group of order $2n$, resp.\ the alternating group of order $n!/2$,
resp.\ the symmetric group of order $n!$). We will also encounter $\Gg_{16}$, a
group of $16$ elements that is a direct product $\CG_4 \times \DG_4$,
$\Gg_{48}$, a group of $48$ elements that is a central extension of $\AG_4$ by
$\CG_4$, $\Gg_{96}$, a group of $96$ elements that is a semidirect product
$(\CG_4 \times \CG_4) \rtimes \SG_3$ and $\Gg_{168}$, which is a group of $168$
elements isomorphic to $\PSL_2 (\FF_7)$.\medskip

\subsection*{Acknowledgments}
We would like to thank Jonas Bergstr\"om, Bas Edixhoven, Everett Howe, Frans
Oort and Matthieu Romagny for their generous help during the writing of this
paper. Also, we warmly thank the anonymous referees for carefully reading this
work and for suggestions.

\section{Families of curves}
\label{sec:famdefs}

Let $g>1$ be an integer, and let $k$ be a field of characteristic $p = 0$ or
$p>2g+1$.  For $\sS$ a scheme over $k$, we define a \emph{curve of genus $g$}
over $\sS$ to be a morphism of schemes $\cc \rightarrow \sS$ that is proper and
smooth with geometrically irreducible fibers of dimension $1$ and genus $g$.
Let $\Mm_g$ be the coarse moduli space of curves of genus $g$ whose geometric
points over algebraically closed extensions $K$ of $k$ correspond with the
$K$-isomorphism classes of curves $C$ over $K$.

We are interested in studying the subvarieties of $\Mm_g$ where the
corresponding curves have an automorphism group isomorphic with a given group.
The subtlety then arises that these subvarieties are not necessarily
irreducible. This problem was also mentioned and studied in~\cite{magaard}, and
resolved by using Hurwitz schemes; but in this section we prefer another way
around the problem, due to L\o{}nsted in~\cite{lonsted}.

In~\cite[Sec.6]{lonsted} the moduli space $\Mm_g$ is stratified in a finer way,
namely by using `rigidified actions' of automorphism groups. Given an
automorphism group $G$, L\o{}nsted defines subschemes of $\Mm_g$ that we shall
call \emph{strata}. Let $\ell$ be a prime different from $p$, and let
$\Gamma_{\ell} = \Sp_{2g} (\FF_{\ell})$. Then the points of a given stratum $S$
correspond to those curves $C$ for which the induced embedding of $G$ into the
group ($\cong \Gamma_{\ell}$) of polarized automorphisms of $\Jac(C)[\ell]$ is
$\Gamma_{\ell}$-conjugate to a given group. Combining~\cite[Th.1]{homma}
with~\cite[Th.6.5]{lonsted} now shows that under our hypotheses on $p$, such a
stratum is a locally closed, connected and smooth subscheme of $\Mm_g$.  If $k$
is perfect, such a connected stratum is therefore defined over $k$ if only one
rigidification is possible for a given abstract automorphism group. As was also
observed in~\cite{magaard}, this is not always the case; and as we will see in
Remark~\ref{rem:conj}, in the case of plane quartics these subtleties are only
narrowly avoided.

We return to the general theory. Over the strata $S$ of $\Mm_g$ with
non-trivial automorphism group, the usual notion of a universal family (as
in~\cite[p.25]{newstead}) is of little use.  Indeed, no universal family can
exist on the non-trivial strata; by~\cite[Sec.14]{abramovich}, $S$ is a fine
moduli space (and hence admits a universal family) if and only if the
automorphism group is trivial. In the definition that follows, we weaken this
notion to that of a \emph{representative family}. While such families coincide
with the usual universal family on the trivial stratum, it will turn out (see
Theorem~\ref{thm:smallstrata}) that they can also be constructed for the strata
with non-trivial automorphism group.  Moreover, they still have sufficiently
strong properties to enable us to effectively parametrize the moduli space.

\begin{definition}\label{def:arith}
  Let $S \subset \Mm_g$ be a subvariety of $\Mm_g$ that is defined over $k$.
  Let $\cc \rightarrow \sS$ be a family of curves whose geometric fibers
  correspond to points of the subvariety $S$, and let $f_{\cc} : \sS \to S$ be
  the associated morphism.
  \begin{enumerate}
    \item The family $\cc \rightarrow \sS$ is \emph{geometrically surjective}
      (for $S$) if the map  $f_{\cc}$ is surjective on $K$-points for every
      algebraically closed extension $K$ of $k$.
    \item The family $\cc \rightarrow \sS$ is \emph{arithmetically surjective}
      (for $S$) if the map  $f_{\cc}$ is surjective on $k'$-points for every
      finite extension $k'$ of $k$.
    \item The family $\cc \rightarrow \sS$ is \emph{quasifinite}
      (for $S$) if it is geometrically surjective and $f_{\cc}$ is quasifinite.
    \item The family $\cc \rightarrow \sS$ is \emph{representative} (for $S$) if
      $f_{\cc}$ is bijective on $K$-points for every algebraically closed
      extension $K$ of $k$.
  \end{enumerate}
\end{definition}

\begin{remark}\label{rem:EGA}
  A family $\cc \rightarrow \sS$ is geometrically surjective if and only if the
  corresponding morphism of schemes $\sS \to S$ is surjective.

  Due to inseparability issues, the morphism $f_{\cc}$ associated to a
  representative family need not induce bijections on points over arbitrary
  extensions of $k$.

  Note that if a representative family $\sS$ is absolutely irreducible, then
  since $S$ is normal, we actually get that $f_{\cc}$ is an isomorphism by
  Zariski's Main Theorem.  However, there are cases where we were unable to
  find such an $\sS$ given a stratum $S$ (see Remark~\ref{rem:nonfin2}).

  The notions of being geometrically surjective, quasifinite and representative
  are stable under extension of the base field $k$. On the other hand, being
  arithmetically surjective can strongly depend on the base field, as for
  example in Proposition~\ref{prop:c2}.
\end{remark}

To prove that quasifinite families exist, one typically considers the universal
family over $\Mm_g^{(\ell)}$ (the moduli space of curves of genus $g$ with full
level-$\ell$ structure, for a prime $\ell>2$ different from $p$,
see~\cite[Th.13.2]{abramovich}). This gives a quasifinite family over $\Mm_g$
by the forgetful (and in fact quotient) map $\Mm_g^{(\ell)} \to \Mm_g$ that we
will denote $\pi_{\ell}$ when using it in our constructions below.

Let $K$ be an algebraically closed extension of $k$. Given a curve $C$ over
$K$, recall that an intermediate field $k \subset L \subset K$ is a \emph{field
of definition} of $C$ if there exists a curve $C_0/L$ such that $C_0$ is
$K$-isomorphic to $C$. The concept of representative families is related with
the question of whether the \emph{field of moduli} $\bfM_C$ of the curve $C$,
which is by definition the intersection of the fields of definition of $C$, is
itself a field of definition. Since we assumed that $p>2g+1$ or $p = 0$, the
field $\bfM_C$ then can be recovered more classically as the residue field of
the moduli space $\Mm_g$ at the point $[C]$ corresponding to $C$
by~\cite[Cor.1.11]{seki}. This allows us to prove the following.

\begin{proposition}\label{prop:desrep}
  Let $S$ be a subvariety of $\Mm_g$ defined over $k$ that admits a
  representative family $\cc \to \sS$. Let $C$ be a curve over an algebraically
  closed extension $K$ of $k$ such that the point $[C]$ of $\Mm_g (K)$ belongs
  to $S$. Then $C$ descends to its field of moduli $\bfM_C$. In case $k$ is
  perfect and $K = \kbar$, then $C$ even corresponds to an element of $\sS
  (\bfM_C)$.
\end{proposition}

\begin{proof}
  First we consider the case where $k = \bfM_C$ and $K$ is a Galois extension
  of $k$. Let $x \in \sS (K)$ be the preimage of $[C]$ under $f_{\cc}$. For
  every $\sigma \in \Gal (K / k)$ it makes sense to consider $x^{\sigma} \in
  \sS(K)$, since the family $\cc$ is defined over $k$. Now since $f_{\cc}$ is
  defined over $k$, we get $f_{\cc}(x)=f_{\cc}(x^{\sigma})=s$. By uniqueness of
  the representative in the family, we get $x=x^{\sigma}$. Since $\sigma$ was
  arbitrary and $K / k$ is Galois, we therefore have $x \in \sS(k)$, which
  gives a model for $C$ over $k$ by taking the corresponding fiber for the
  family $\cc \to \sS$. This already proves the final statement of the
  proposition.

  Since the notion of being representative is stable under changing the base
  field $k$, the argument in the Galois case gives us enough leverage to treat
  the general case (where $K / k$ is possibly transcendental or inseparable) by
  appealing to~\cite[Th.1.6.9]{huggins-thesis}.
\end{proof}

Conversely, we have the following result. A construction similar to it will be
used in the proof of Theorem~\ref{thm:smallstrata}.

\begin{proposition}\label{prop:iter}
  Let $S$ be a stratum defined over a field $k$. Suppose that for every finite
  Galois extension $F \supset E$ of field extensions of $k$, the field of
  moduli of the curve corresponding to a point in $S(E)$ equals $E$.  Then
  there exists a representative family $\cc_U \to U$ over a dense open subset
  of $S$. If $k$ is perfect, this family extends to a possibly disconnected
  representative family $\cc \to \sS$ for the stratum $S$.
\end{proposition}

\begin{proof}
  Let $\eta$ be the generic point of $S$ and again let $\pi_{\ell} :
  \Mm_g^{(\ell)} \to \Mm_g$ be the forgetful map obtained by adding level
  structure at a prime $\ell>2$ different from $p$. Note that as a quotient by
  a finite group, $\pi_{\ell}$ is a finite Galois cover. Let $\nu$ be a generic
  point in the preimage of $\eta$ by $\pi_{\ell}$ and $\cc \to \nu$ be the
  universal family defined over $k(\nu)$. By definition the field of moduli
  $\bfM_{\cc}$ is equal to $k(\nu)$ and as $k(\nu)$ is a field of definition
  there exists a family $\cc_0 \to k(\nu)$ geometrically isomorphic to $\cc$.
  Since $k(\nu) \supset k(\eta)$ is a Galois extension, we can argue as in the
  proof of Proposition~\ref{prop:desrep} to descend to $k(\eta)$, and hence by
  a spreading-out argument we can conclude that $\cc_0$ is a representative
  family on a dense open subset $U$ of $S$. Proceeding by induction over the
  (finite) union of the Galois conjugates of the finitely many irreducible
  components of the complement of $U$, which is again defined over $k$, one
  obtains the second part of the proposition.
\end{proof}

Whereas the universal family $\cc \to \Mm_g^{(\ell)}$ is sometimes easy to
construct, it seems hard to work out $\cc_0$ directly by explicit Galois
descent; the Galois group of the covering $\Mm_g^{(\ell)} \to \Mm_g$ is
$\Sp_{2g}(\FF_{\ell})$, which is a group of large cardinality $\ell^{g^2}
\prod_{i=1}^g (\ell^{2i}-1)$ whose quotient by its center is simple. Moreover,
for enumeration purposes, it is necessary for the scheme $\sS$ to be as simple
as possible.  Typically one would wish for it to be rational, as fortunately
turns out always to be the case for plane quartics. On the other hand, for
moduli spaces of \emph{general type} that admit no rational curves, such as
$\Mm_g$ with $g> 23$, there does not even exist a rational family of curves
with a single parameter~\cite{harris-mumford}.

\section{Families of smooth plane quartics}
\label{sec:famquart}

\subsection{Review : automorphism groups}
\label{subsec:review}

Let $C$ be a smooth plane quartic over an algebraically closed field $K$ of
characteristic $p \geq 0$. Then since $C$ coincides up to a choice of basis
with its canonical embedding, the automorphism $\Aut (C)$ can be considered as
a conjugacy class of subgroups $\PGL_3 (K)$ (and in fact of $\GL_3 (K)$) by
using the action on its non-zero differentials.

The classification of the possible automorphism groups of $C$ as subgroup of
$\PGL_3(K)$, as well as the construction of some geometrically complete
families, can be found in several articles, such
as~\cite[2.88]{henn},~\cite[p.62]{vermeulen},~\cite{magaard},~\cite{bars}
and~\cite{dolgacag} (in chronological order), in which it is often assumed that
$p = 0$.  We have verified these results independently, essentially by checking
which finite subgroups of $\PGL_3 (K)$ (as classified
in~\cite[Lem.2.3.7]{huggins-thesis}) can occur for plane quartics. It turns out
that the classification in characteristic $0$ extends to algebraically closed
fields $K$ of prime characteristic $p \geq 5$. In the following theorem, we do
not indicate the open non-degeneracy conditions on the affine parameters, since
we shall not have need of them.

\begin{theorem}\label{thm:aut}
  Let $K$ be an algebraically closed field whose characteristic $p$ satisfies
  $p = 0$ or $p \geq 5$. Let $C$ be a genus $3$ non-hyperelliptic curve over
  $K$. The following are the possible automorphism groups of $C$, along with
  geometrically surjective families for the corresponding strata:
  \begin{enumerate}
  \item $ \{1\}$, with family $q_4(x,y,z)=0,$ where $q_4$ is a homogeneous
    polynomial of degree $4$;\vspace*{5pt}
  \item $ \CG_2$, with family $x^4 + x^2 q_2(y,z) + q_4(y,z) = 0$, where $q_2$
    and $q_4$ are homogeneous polynomials in $y$ and $z$ of degree $2$ and
    $4$;\vspace*{5pt}
  \item $ \DG_4$, with family $x^4 + y^4 + z^4 + r x^2 y^2 + s y^2 z^2 + t z^2
    x^2 = 0$;\vspace*{5pt}
  \item $\CG_3$, with family $x^3 z + y (y - z) (y - r z) (y - s z)=0$;
  \item $ \DG_8$, with family $x^4 + y^4 + z^4 + r x^2 y z + s y^2 z^2 = 0$;
  \item $\SG_3$, with family $x (y^3 + z^3) + y^2 z^2 + r x^2 y z + s x^4 =
    0$;\vspace*{5pt}
  \item $ \CG_6$, with family $x^3 z + y^4 + r y^2 z^2 + z^4 = 0$;
  \item $ \Gg_{16}$, with family $x^4 + y^4 + z^4 + r y^2 z^2 = 0$;
  \item $ \SG_4$, with family $x^4 + y^4 + z^4 + r (x^2 y^2 + y^2 z^2 + z^2
    x^2) = 0$;\vspace*{5pt}
  \item $ \CG_9$, represented by the quartic $x^3 y + y^3 z + z^4 = 0$;
  \item $ \Gg_{48}$, represented by the quartic $x^4 + (y^3 - z^3) z = 0$;
  \item $ \Gg_{96}$, represented by the Fermat quartic $x^4 + y^4 + z^4 = 0$;
  \item (if $p \neq 7$) $ \Gg_{168}$, represented by the Klein quartic $x^3 y +
    y^3 z + z^3 x = 0$.
  \end{enumerate}
\end{theorem}

The families in Theorem~\ref{thm:aut} are geometrically surjective. Moreover,
they are irreducible and quasifinite (as we will see in the proof of
Theorem~\ref{thm:smallstrata}) for all groups except the trivial group and
$\CG_2$.
The embeddings of the automorphism group of these curves into $\PGL_3 (K)$ can
be found in Theorem~\ref{thm:norm} in Appendix~\ref{sec:gener-norm}. Because of
the irreducibility properties mentioned in the previous paragraph, each of the
corresponding subvarieties serendipitously describes an actual stratum in the
moduli space $\Mm_3^{\textrm{nh}} \subset \Mm_3$ of genus $3$ non-hyperelliptic
curves as defined in Section~\ref{sec:famdefs} (see Remark~\ref{rem:conj}
below). From the descriptions in~\ref{thm:norm}, one derives the inclusions
between the strata indicated in Figure~\ref{fig:strata}, as also obtained
in~\cite[p.65]{vermeulen}.

\begin{figure}[htbp]
  \centering
  \resizebox{0.6\textwidth}{!}{
    $\xymatrix@R-=10pt{
      \dim 6 & & & & \ar@{-}[dddll] \ar@{-}[d] \{1\}  &  & \\
      \dim 4 & & & & \CG_2 \ar@{-}[d] \ar@{-}[dddll] \ar@{-}[ddrr] & & \\
      \dim 3 & & & & \DG_4 \ar@{-}[d] & & \\
      \dim 2 & & \CG_3 \ar@{-}[ddl] \ar@{-}[d] & & \DG_8 \ar@{-}[dl]
      \ar@{-}[dr]&
      & \SG_3 \ar@{-}[dl] \\
      \dim 1 & & \CG_6 \ar@{-}[d] & \Gg_{16} \ar@{-}[dl] \ar@{-}[dr]&
      & \SG_4 \ar@{-}[dl] \ar@{-}[dr] \\
      \dim 0 & \CG_9 & \Gg_{48} & & \Gg_{96} & & \Gg_{168} \\
    }$
  }
  \caption{Automorphism groups}
  \label{fig:strata}
\end{figure}
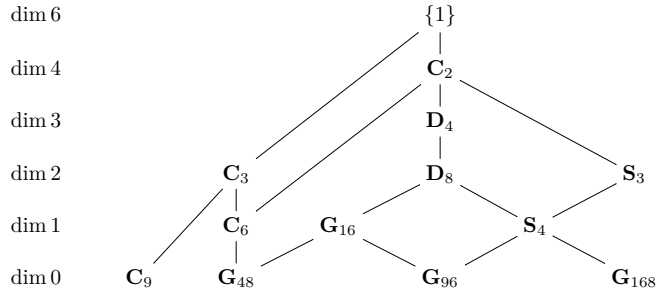

\begin{remark}\label{rem:conj}
  As promised at the beginning of Section~\ref{sec:famdefs}, we now indicate
  two different possible rigidifications of an action of a finite group on
  plane quartics. Consider the group $\CG_3$. Up to conjugation, this group can
  be embedded into $\PGL_3 (K)$ in exactly two ways; as a diagonal matrix with
  entries proportional to $(\zeta_3 , 1 , 1)$ or $(\zeta_3^2 , \zeta_3 , 1)$.
  This gives rise to two rigidifications in the sense of L\o{}nsted.

  While for plane curves of sufficiently high degree, this indeed leads to two
  families with generic automorphism group $\CG_3$, the plane quartics
  admitting the latter rigidification always admit an extra involution, so that
  the full automorphism group contains $\SG_3$. It is this fortunate phenomenon
  that still makes a naive stratification by automorphism groups possible for
  plane quartics. For the same reason, the stratum for the group $\SG_3$ is not
  included in that for $\CG_3$, as is claimed incorrectly in~\cite{bars}.
\end{remark}

\subsection{Construction of representative families}
\label{subsec:repfam}

We now describe how to apply Galois descent to extensions of function fields to
determine representative families for the strata in Theorem~\ref{thm:aut} with
$|G| > 2$. By Proposition~\ref{prop:desrep}, this shows that the descent
obstruction always vanishes for these strata.

Our constructions lead to families that parametrize the strata much more
efficiently; for the case $\DG_4$, the family in Theorem~\ref{thm:aut} contains
as much as $24$ distinct fibers isomorphic with a given curve.  Moreover, by
Proposition~\ref{prop:desrep}, in order to write down a complete list of the
$\kbar$-isomorphism classes of smooth plane quartics defined over a perfect
field $k$ we need only consider the $k$-rational fibers of the new families.

As in Theorem~\ref{thm:aut}, we do not specify the condition on the parameters
that avoid degenerations (\ie singular curves or a larger automorphism group),
but such degenerations will be taken into account in our enumeration strategy
in Section~\ref{sec:impl-exper}.

\begin{theorem}\label{thm:smallstrata}
  Let $k$ be a field whose characteristic $p$ satisfies $p = 0$ or $p \geq 7$.
  The following are representative families for the strata of smooth plane
  quartics with $| G | > 2$.
  \begin{itemize}
  \item $G \simeq \DG_4$:
    \begin{dmath*}[compact,style={\footnotesize},spread={-3pt}]
      (a + 3) x^4 + (4 a^2 - 8 b + 4 a) x^3 y + (12 c + 4 b) x^3 z + (6 a^3 - 18
      a b + 18 c + 2 a^2) x^2 y^2 + (12 a c + 4 a b) x^2 y z + (6 b c + 2 b^2)
      x^2 z^2 + (4 a^4 - 16 a^2 b + 8 b^2 + 16 a c + 2 a b - 6 c) x y^3 + (12
      a^2 c - 24 b c + 2 a^2 b - 4 b^2 + 6 a c) x y^2 z + (36 c^2 + 2 a b^2 - 4
      a^2 c + 6 b c) x y z^2 + (4 b^2 c - 8 a c^2 + 2 a b c - 6 c^2) x z^3 +
      (a^5 - 5 a^3 b + 5 a b^2 + 5 a^2 c - 5 b c + b^2 - 2 a c) y^4 + (4 a^3 c -
      12 a b c + 12 c^2 + 4 a^2 c - 8 b c) y^3 z + (6 a c^2 + a^2 b^2 - 2 b^3 -
      2 a^3 c + 4 a b c + 9 c^2) y^2 z^2 + (4 b c^2 + 4 b^2 c - 8 a c^2) y z^3 +
      (b^3 c - 3 a b c^2 + 3 c^3 + a^2 c^2 - 2 b c^2) z^4  = 0
    \end{dmath*}
  along with
    \begin{dmath*}[compact,style={\footnotesize},spread={-3pt}]
      x^4 + 2 x^2 y^2 + 2 a x^2 y z + (a^2 - 2 b) x^2 z^2 + a y^4 + 4 (a^2 - 2
      b) y^3 z + 6 (a^3 - 3 a b) y^2 z^2 + 4 (a^4 - 4 a^2 b + 2 b^2) y z^3 +
      (a^5 - 5 a^3 b + 5 a b^2) z^4 = 0  .
    \end{dmath*}
    \vspace*{5pt}
  \item $G \simeq \CG_3$: $x^3 z + y^4 + a y^2 z^2 + a y z^3 + b z^4 = 0$
    along with $x^3 z + y^4 + a y z^3 + a z^4 = 0$;
  \item $G \simeq \DG_8$: $x^4 + x^2 y z + y^4 + a y^2 z^2 + b z^4 = 0$;
  \item $G \simeq \SG_3$: $x^3 z + y^3 z + x^2 y^2 + a x y z^2 + b z^4 = 0$;
    \vspace*{5pt}
  \item $G \simeq \CG_6$: $x^3 z + a y^4 + a y^2 z^2 + z^4 = 0$;
  \item $G \simeq \Gg_{16}$: $x^4 + (y^3 + a y z^2 + a z^3) z = 0$;
  \item $G \simeq \SG_4$: $x^4 + y^4 + z^4 + a (x^2 y^2 + y^2 z^2 + z^2 x^2) =
    0$;
    \vspace*{5pt}
  \item $G \simeq \CG_9$: $x^3 y + y^3 z + z^4 = 0$;
  \item $G \simeq \Gg_{48}$: $x^4 + (y^3 - z^3) z = 0$;
  \item $G \simeq \Gg_{96}$: $x^4 + y^4 + z^4 = 0$;
  \item (if $p \neq 7$) $G \simeq \Gg_{168}$: $x^3 y + y^3 z + z^3 x = 0$.
  \end{itemize}
\end{theorem}

We do not give the full proof of this theorem, but content ourselves with some
families that illustrate the most important ideas therein. Let $K$ be an
algebraically closed extension of $k$. The key fact that we use, which can be
observed from the description in Theorem~\ref{thm:norm}, is that the fibers of
the families in Theorem~\ref{thm:aut} all have the same automorphism group $G$
as a subgroup of $\PGL_3 (K)$. Except for the zero-dimensional cases, which are
a one-off verification, one then proceeds as follows.

\begin{enumerate}
  \item The key fact above implies that any isomorphism between two curves in
    the family is necessarily induced by an element of the normalizer $N$ of
    $G$ in $\PGL_3 (K)$. So one considers the action of this group on the
    family given in Theorem~\ref{thm:aut}.
  \item One determines the subgroup $N'$ of $N$ that sends the family to itself
    again. The action of $N'$ factors through a faithful action of $Q = N' /
    G$. By explicit calculation, it turns out that $Q$ is finite for the
    families in Theorem~\ref{thm:aut} with $|G| > 2$. This shows in particular
    that these families are already quasifinite on these strata.
  \item One then takes the quotient by the finite action of $Q$, which is
    done one the level of function fields over $K$ by using Galois descent. By
    construction, the resulting family will be representative. For the general
    theory of Galois descent, we refer to~\cite{wei56}
    and~\cite[App.A]{varshavsky}.
\end{enumerate}

We now treat some representative cases to illustrate this procedure. In what
follows, we use the notation from Theorem~\ref{thm:norm} to denote elements and
subgroups of the normalizers involved.
%\smallskip

\begin{proof}
  \textbf{The case $G \simeq \SG_3.$} Here $N = T(K) \widetilde{\SG}_3$
  contains the group of diagonal matrices $T(K)$. Transforming, one verifies
  that the subgroup $N' \subsetneq N$ equals $\widetilde{\SG_3}$; indeed, since
  $\widetilde{\SG}_3$ fixes the family pointwise, we can restrict to the
  elements $T(K)$. But then preserving the trivial proportionality of the
  coefficients in front of $x^3 z$, $y^3 z$, and $x^2 y^2$ forces such a
  diagonal matrix to be scalar. This implies the result; the group $Q$ is
  trivial in $\PGL_3 (K)$, so we need not adapt our old family since it is
  geometrically surjective and contains no geometrically isomorphic fibers. A
  similar argument works for the case $G \cong \SG_4$.  \smallskip

  \textbf{The case $G \simeq \CG_6$.} This time we have to consider the action
  of the group $D (K)$ on the family $x^3 z + y^4 + r y^2 z^2 + z^4 = 0$ from
  Theorem~\ref{thm:aut}.  After the action of a diagonal matrix with entries
  $\lambda, \mu, 1$, one obtains the curve $\lambda^3 x^3 z + \mu^4 y^4 + \mu^2
  r y^2 z^2 + z^4 = 0$. We see that we get a new curve in the family if
  $\lambda^3 = 1$ and $\mu^4 = 1$, in which case the new value for $r$ equals
  $\mu r$. But this equals $\pm r$ since $(\mu^2)^2 = 1$. The degree of the
  morphism to $\Mm_3$ induced by this family therefore equals $2$. This also
  follows from the fact that the subgroup $N'$ that we just described contains
  $G$ as a subgroup of index $2$, so that $Q \cong \CG_2$.

  We have a family over $L = K (r)$ whose fibers over $r$ and $-r$ are
  isomorphic, and we want to descend this family to $K(a)$, where $a = r^2$
  generates the invariant subfield under the automorphism $r \rightarrow -r$.
  This is a problem of Galois descent for the group $Q \cong \CG_2$ and the
  field extension $M \supset L$, with $M = K (r)$ and $L = K(a)$. The curve $C$
  over $M$ that we wish to descend to $L$ is given by $x^3 z + y^4 + r y^2 z^2
  + z^4 = 0$. Consider the conjugate curve $C^{\sigma} : x^3 z + y^4 - r y^2
  z^2 + z^4 = 0$ and the isomorphism $\varphi : C \to C^{\sigma}$ given by
  $(x,y,z) \to (x,i y,z)$. Then we do not have $\varphi^{\sigma} \varphi =
  \Id$. To trivialize the cocycle, we need a larger extension of our function
  field $L$.

  Take $M' \supset M$ to be $M' = M(\rho)$, with $\rho^2 = r$.  Let $\tau$ be a
  generator of the cyclic Galois group of order $4$ of the extension $M'
  \supset L$. Then $\tau$ restricts to $\sigma$ in the extension $M \supset L$,
  and for $M' \supset L$ one now indeed obtains a Weil cocycle determined by
  the isomorphism $C \mapsto C^{\tau} = C^{\sigma}$ sending $(x,y,z)$ to $(x,i
  y,z)$. The corresponding coboundary is given by $(x,y,z) \mapsto ( x , \rho y
  , z)$. Transforming, we end up with
  \begin{math}
    x^3 z + (\rho y)^4 + r (\rho y)^2 z^2 + z^4  = x^3 z + a y^4 + a y^2 z^2 + z^4
    = 0,
  \end{math}
  which is what we wanted to show. The case $G \cong \DG_8$ can be dealt with
  in a similar way.\smallskip

  \textbf{The case $G \simeq \DG_4$.} We start with the usual \emph{Ciani
  family} from Theorem~\ref{thm:aut}, given by
  \begin{math}
    x^4 + y^4 + z^4 + r x^2 y^2 + s y^2 z^2 + t z^2 x^2  = 0 .
  \end{math}
  Using the $\widetilde{\SG}_3$-elements from the normalizer $N = D(K)
  \widetilde {\SG}_3$ induces the corresponding permutation group on $(r,s,t)$.
  The diagonal matrices in $D(K)$ then remain, and they give rise to the
  transformations $(r,s,t) \mapsto (\pm r , \pm s , \pm t)$ with an even number
  of minus signs. This is slightly awkward, so we try to eliminate the latter
  transformations. This can be accomplished by moving the parameters in front
  of the factors $x^4$, $y^4$, $z^4$. So we instead split up $\sS$ into a
  disjoint union of two irreducible subvarieties by considering the family
  \begin{align*}
    r x^4 + s y^4 + t z^4 + x^2 y^2 + y^2 z^2 + z^2 x^2   = 0 ,
  \end{align*}
  and its lower-dimensional complement
  \begin{align*}
    r x^4 + s y^4 + z^4 + x^2 y^2 + y^2 z^2 = 0 .
  \end{align*}
  Here the trivial coefficient in front of $z^4$ is obtained by scaling $x,y,z$
  by an appropriate factor in the family $r x^4 + s y^4 + t z^4 + x^2 y^2 + y^2
  z^2 = 0$. Note that because of our description of the normalizer, the number
  of non-zero coefficients in front of the terms with quadratic factors depends
  only on the isomorphism class of the curve, and not on the given equation for
  it in the geometrically surjective Ciani family.  This implies that the two
  families above do not have isomorphic fibers. Moreover, the a priori
  remaining family $r x^4 + y^4 + z^4 + y^2 z^2 = 0$ has larger automorphism
  group, so we can discard it.

  We only consider the first family, which is the most difficult case. As in
  the previous example, after our modification the elements of $N' \cap D(K)$
  are in fact already in $G$. Therefore $Q = N' / G \subset D(K)
  \widetilde{\SG}_3$ is a quotient of the remaining factor $\widetilde{\SG}_3$,
  which clearly acts freely and is therefore isomorphic with $Q$.  We obtain
  the invariant subfield $L = K(a,b,c)$ of $M = K (r,s,t)$, with $a = r + s +
  t$, $b = r s + s t + t r$ and $c = r s t$ the usual elementary symmetric
  functions. The cocycle for this extension is given by sending a permutation
  of $(r,s,t)$ to its associated permutation matrix on $(x,y,z)$. A coboundary
  is given by the isomorphism
  \begin{math}
    (x,y,z)\mapsto (x+y+z, r\,x+s\,y+t\,z, st\,x+tr\,y+rs\,z)\,.
  \end{math}
  Note that this isomorphism is invertible as long as $r,s,t$ are distinct,
  which we may assume since otherwise the automorphism group of the curve would
  be larger. Transforming by this coboundary, we get our result.\smallskip

  \textbf{The case $G \simeq \CG_3$.} This case needs a slightly different
  argument. Consider the eigenspace decomposition of the space of quartic
  monomials in $x,y,z$ under the action of the diagonal generator
  $(\zeta_3,1,1)$ of $\CG_3$. The curves with this automorphism correspond to
  those quartic forms that are eigenforms for this automorphism, which is the
  case if and only if it is contained in one of the aforementioned eigenspaces.
  We only need consider the eigenspace spanned by the monomials $x^3 y$, $x^3
  z$, $y^4$, $y^3 z$, $y^2 z^2$, $y z^3$, $z^4$; indeed, the quartic forms in
  the other eigenspaces are all multiples of $x$ and hence give rise to
  reducible curves.

  Using a linear transformation, one eliminates the term with $x^3 y$, and a
  non-singularity argument shows that we can scale to the case
  \begin{math}
    x^3 z + y^4 + r y^3 z + s y^2 z^2 + t y z^3 + u z^4 = 0  .
  \end{math}
  We can set $r = 0$ by another linear transformation, which then reduces $N'$
  to $D (K)$. Depending on whether $s = 0$ or not, one can then scale by these
  scalar matrices to an equation as in the theorem, which one verifies to be
  unique by using the same methods as above. The case $G \simeq \Gg_{16}$ can
  be proved in a completely similar way.
\end{proof}

\begin{remark}\label{rem:nonfin2}
  As mentioned in Remark~\ref{rem:EGA}, these constructions give rise to
  isomorphisms $\sS \to S$ in all cases except $\DG_4$, and $\CG_3$. In these
  remaining cases, we have constructed a morphism $\sS \to S$ that is bijective
  on points but not an isomorphism. It is possible that no family $\cc \to S$
  inducing such an isomorphism exists; see~\cite{gorvivi} for results in this
  direction for hyperelliptic curves.
\end{remark}

\subsection{Remaining cases}
\label{sec:big}

We have seen in Proposition~\ref{prop:desrep} that if there exist a
representative family over $k$ over a given stratum, then the field of moduli
needs to be a field of definition for all the curves in this stratum.
In~\cite{artqui}, it is shown that there exist $\RR$-points in the stratum
$\CG_2$ for which the corresponding curve cannot be defined over $\RR$.  In
fact we suspect that this argument can be adapted to show that representative
families for this stratum fail to exist even if $k$ is a finite field. However,
we can still find arithmetically surjective families over finite fields.

\begin{proposition} \label{prop:c2}
  Let $C$ be a smooth plane quartic with automorphism group $\CG_2$ over a
  finite field $k$ of characteristic different from $2$. Let $\alpha$ be a
  non-square element in $k$. Then $C$ is $k$-isomorphic to a curve of one of
  the following forms:
  {\small
  \begingroup
  \addtolength{\jot}{-3pt}
  \begin{align*}
    & x^4+\epsilon x^2 y^2 + a y^4+\mu y^3 z+b y^2 z^2+c y z^3+d z^4 = 0 &
      \textrm{with} \; \epsilon=1 \; \textrm{or} \; \alpha \; \textrm{and}  \;
      \mu=0 \; \textrm{or} \; 1, \\
    & x^4+x^2 y z + a y^4+\epsilon y^3 z+b y^2 z^2+c y z^3+d z^4 = 0 &
    \textrm{with} \; \epsilon=0,1 \; \textrm{or} \; \alpha, \\
    & x^4+x^2 (y^2-\alpha z^2) + a y^4+b y^3 z+c y^2 z^2+d y z^3+e z^4 = 0 \,.
    &
  \end{align*}
  \endgroup
  }
\end{proposition}

\begin{proof}
  The involution on the quartic, being unique, is defined over $k$. Hence by
  choosing a basis in which this involution is a diagonal matrix, we can assume
  that it is given by $(x,y,z) \mapsto (-x,y,z)$. This shows that the family
  $x^4+ x^2 q_2(y,z) + q_4(y,z) = 0$ of Theorem~\ref{thm:aut} is arithmetically
  surjective. We have $q_2(y,z) \ne 0$ since otherwise more automorphisms would
  exist over $K$. We now distinguish cases depending on the factorization of
  $q_2$ over $k$.
  \begin{enumerate}
    \item If $q_2$ has a multiple root, then we may assume that $q_2(y,z)=r y^2$
      where $r$ equals $1$ or $\alpha$. Then either the coefficient $b$ of $y^3
      z$ in $q_4$ is $0$, in which case we are done, or we can normalize it to
      $1$ using the change of variable $z \mapsto z/b $.
    \item If $q_2$ splits over $k$, then we may assume that $q_2(y,z)=yz$. Then
      either the coefficient $b$ of $y^3 z$ in $q_4$ is $0$, in which case we
      are done, or we attempt to normalize it by a change of variables $y \mapsto
      \lambda y$ and $z \mapsto z/\lambda $. This transforms $b y^3 z$ into $b
      \lambda^2 y^3 z$. Hence we can assume $b$ equals $1$ or $\alpha$.
    \item If $q_2$ is irreducible over $k$, then we can normalize $q_2(y,z)$ as
      $y^2- \alpha z^2$ where $\alpha$  is a non-square in $k$. This gives us
      the final family with $5$ coefficients. \qedhere
  \end{enumerate}
\end{proof}

\begin{remark}
  The same proof shows the existence of a quasifinite family for the stratum in
  Proposition~\ref{prop:c2}, since over algebraically fields we can always
  reduce to the first or second case.
\end{remark}

We have seen in Section~\ref{sec:famdefs} that a universal family exists for
the stratum with trivial automorphism group.  Moreover, as $\Mm_3$ is rational
\cite{katsylo}, this family depends on $6$ rational parameters. However, no
representative (hence in this case universal) family seems to have been written
down so far.

Classically, when the characteristic $p$ is different from $2$ or $3$, there
are at least two ways to construct quasifinite families for the generic
stratum.  The first method fixes bitangents of the quartic and leads to the
so-called Riemann model; see \cite{grossharris,agmri,weber} for relations
between this construction, the moduli of $7$ points in the projective plane and
the moduli space $\Mm_3^{(2)}$. The other method uses flex points, as in
\cite[Prop.1]{shioda-quartic}.  In neither case can we get such models over the
base field $k$, since for a general quartic, neither its bitangents nor its
flex points are defined over $k$.
We therefore content ourselves with the following result which was kindly
provided to us by J. Bergstr\"om.

\begin{proposition}[Bergstr\"om]\label{prop:bergstrom}
  Let $C$ be a smooth plane quartic over a field $k$ admitting a rational point
  over a field of characteristic $\neq 2, 3$.  Then $C$ is isomorphic to a
  curve of one of the following forms: {\small
  \begingroup
  \addtolength{\jot}{-3pt}
  \begin{align*}
    m_1 x^4 + m_2 x^3 y + m_4 x^2 y^2 + m_6 x^2 z^2 + m_7 x y^3 + x y^2 z +
      m_{11} y^4 + m_{12} y^3 z + y^2 z^2 + y z^3 = 0,  \\
    m_1 x^4 + m_2 x^3 y + m_4 x^2 y^2 + m_6 x^2 z^2 + x y^3 + m_{11} y^4 +
      m_{12} y^3 z + y^2 z^2 + y z^3 = 0, \\
    m_1 x^4 + m_2 x^3 y + m_4 x^2 y^2 + m_6 x^2 z^2 + m_{11} y^4 + m_{12} y^3 z
      + y^2 z^2 + y z^3 = 0, \\
    m_1 x^4 + m_2 x^3 y + m_4 x^2 y^2 + m_6 x^2 z^2 + x y^3 + x y^2 z + m_{11}
      y^4 + m_{12} y^3 z + y z^3 = 0, \\
    m_1 x^4 + m_2 x^3 y + m_4 x^2 y^2 + m_6 x^2 z^2 + x y^2 z + m_{11} y^4 +
      m_{12} y^3 z + y z^3 = 0, \\
    x^4 + m_2 x^3 y + m_4 x^2 y^2 + m_6 x^2 z^2 + m_7 x y^3 + m_{11} y^4 +
      m_{12} y^3 z + y z^3 = 0, \\
    m_2 x^3 y + m_4 x^2 y^2 + m_6 x^2 z^2 + m_7 x y^3 + m_{11} y^4 + m_{12} y^3
      z + y z^3 = 0, \\
    x^3 z + m_4 x^2 y^2 + m_7 x y^3 + m_8 x y^2 z + x y z^2 + m_{11} y^4 +
      m_{12} y^3 z + m_{13} y^2 z^2 + y z^3 = 0, \\
    x^3 z + m_4 x^2 y^2 + m_7 x y^3 + m_8 x y^2 z + m_{11} y^4 + m_{12} y^3 z +
      m_{13} y^2 z^2 + y z^3 = 0, \\
    x^4 + m_4 x^2 y^2 + m_5 x^2 y z + m_7 x y^3 + m_8 x y^2 z + m_{11} y^4 +
      m_{12} y^3 z + y z^3 = 0.
  \end{align*}
  \endgroup
  }
\end{proposition}

\begin{proof}
  We denote by $m_1,\ldots,m_{15}$ the coefficients of the quartic $C$, with
  its monomials ordered as
  \begin{equation}\label{eq:1}
    x^4,x^3 y,x^3 z,x^2 y^2,x^2 y z,x^2 z^2,x y^3,x y^2 z, x y z^2,x z^3,y^4,y^3
    z,y^2 z^2,y z^3,z^4.
  \end{equation}

  As there is a rational point on the curve, we can transform this point to be
  $(0:0:1)$ with tangent equal to $y=0$. We then have $m_{15}=m_{10}=0$, and we
  can scale to ensure that $m_{14}=1$. The proof now divides into
  cases.\medskip

  \noindent
  {\bf Case 1: $m_6 \ne 0$.} Consider the terms $m_6 x^2 (z^2 + m_3/m_6 x z)$.
  Then by a further change of variables $z \to z+m_3\, x/(2 m_6) $ we can
  assume $m_3=0$ without perturbing the  previous conditions. Starting with
  this new equation, we can now cancel $m_5$ in the same way, and finally $m_9$
  (note that the order in which we cancel the coefficients $m_3,m_5,m_9$ is
  important, so as to avoid re-introducing non-zero coefficients).
  \begin{enumerate}
    \item If $m_8$ and $m_{13}$ are non-zero, then we can ensure that
      $m_{8}=m_{13}=1$ by changing variables $(x:y:z) \to (r x:s y:t z)$ such
      that $m_8 r s^2 t= \alpha, \quad m_{13} s^2 t^2=\alpha, \quad s
      t^3=\alpha$ for a given $\alpha \ne 0$ and then divide the whole equation
      by $\alpha$. One calculates that it is indeed possible to find  a
      solution $(r,s,t,\alpha)$ to these equations in $k^4$.
    \item If $m_8=0 , m_{13} \ne 0 , m_7 \ne 0$, then we can transform to
      $m_{13}=m_7=1$ as above;
    \item If $m_8=0,m_{13} \ne 0 , m_7=0$, then we can transform to $m_{13}=1$;
    \item If $m_8 \ne 0, m_{13} = 0, m_7 \ne 0$, then we can transform to
      $m_8=m_7=1$;
    \item If $m_8 \ne 0 , m_7 = m_{13} = 0$, then we can transform to $m_8=1$;
    \item If $m_{13} = m_8 = 0, m_1 \ne 0$, then we can transform to $m_1=1$;
    \item If $m_{13} = m_8 = m_1 = 0$, then we need not do anything.
  \suspend{enumerate}\medskip

  \noindent
  {\bf Case 2: $m_6 = 0 , m_3 \ne 0$.} As before, working in the correct order
  we can ensure that $m_1=m_2=m_5=0$ by using the non-zero coefficient $m_3$.
  \resume{enumerate}
    \item If $m_9 \ne 0$, we can transform to $m_3=m_9=1$;
    \item If $m_9=0$, we can transform to $m_3=1$.
  \suspend{enumerate}\medskip

  \noindent
  {\bf Case 3: $m_6 = m_3 = 0$.}
  \resume{enumerate}
    \item If $m_1 \ne 0$, then put $m_1=1$. Using $m_{14}$, we can transform to
      $m_9=m_{13}=0$ and using $m_1$, we can transform to $m_2=0$.
  \end{enumerate}
  The proof is now concluded by noting that if $m_1=m_3=m_6=m_{10}=m_{15}=0$,
  then the quartic is reducible.
\end{proof}

Bergstr\"om has also found models when rational points are not available, but
these depend on as many as $9$ coefficients. Using the Hasse-Weil-Serre bound,
one shows that when $k$ is a finite field  with $\# k>29$, the models in
Proposition~\ref{prop:bergstrom} constitute an arithmetically surjective family
of dimension $7$, one more than the dimension of the moduli space.\medskip

Over finite fields $k$ of characteristic $>7$ and with $\# k \leq  29$ there
are always pointless curves~\cite{HLT}. Our experiments showed that except for
one single example, these curves all have non-trivial automorphism group. As
such, they already appear in the non-generic family. The exceptional pointless
curve, defined over $\FF_{11}$, is
{
  \begin{multline*}
    7 x^4 + 3 x^3 y + 10 x^3 z + 10 x^2 y^2 + 10 x^2 y z + 6 x^2 z^2 +7 x y^2 z\\
    + x y z^2 + 4 x z^3 + 9 y^4 + 5 y^3 z + 8 y^2 z^2 + 9 y z^3 + 9 z^4 = 0 \,.
  \end{multline*}
}

\section{Computation of twists}
\label{sec:computation-twists}

Let $C$ be a smooth plane quartic defined over a finite field $k=\mathbb{F}_q$
of characteristic $p$. In this section we will explain how to compute the
\emph{twists} of $C$, \ie the $k$-isomorphism classes of the curves isomorphic
with $C$ over $\kbar$.

Let $\Twist (C)$ be the set of twists of $C$. This set is in bijection with the
cohomology set $H^1(\Gal ( \overline{k}/k ),\Aut( C ))$,
(see~\cite[Chap.X.2]{MR2514094}). More precisely, if $\beta : C' \to C$ is any
$\overline{k}$-isomorphism, the corresponding element in $H^1(\Gal (
\overline{k}/k ),\Aut( C ))$ is given by $\sigma \mapsto \beta^{\sigma}
\beta^{-1}$.  Using the fact that $\Gal(\overline{k}/k)$ is pro-cyclic
generated by the Frobenius morphism $\varphi : x \mapsto x^q$, computing
$H^1(\Gal ( \overline{k}/k ),\Aut( C ))$ boils down to computing the
equivalence classes of $\Aut( C )$ for the relation
\[
  g \sim h  \iff \exists \alpha \in \Aut(C), \; g\alpha=\alpha^{\varphi }h ,
\]
as in~\cite[Prop.9]{MR2678623}. For a representative $\alpha$ of such a
\emph{Frobenius conjugacy class}, there will then exist a curve $C_{\alpha}$
and an isomorphism $\beta : C_{\alpha} \to C$ such that $\beta^{\varphi }
\beta^{-1} = \alpha$.

As isomorphisms between smooth plane quartics are
linear~\cite[6.5.1]{dolgacag}, $\beta$ lifts to an automorphism of $\PP^2$,
represented by an element $B$ of $\GL_3(\overline{k})$, and we will then have
that $C_{\alpha} = B^{-1} (C)$ as subvarieties of $\PP^2$. This is the curve
defined by the equation obtained by substituting $B (x,y,z)^t$ for the
transposed vector $(x,y,z)^t$ in the quartic relations defining $C$.

\subsection{Algorithm to compute the twists of a smooth plane quartic}

We first introduce a probabilistic algorithm to calculate the twists of $C$. It
is based on the explicit form of Hilbert 90 (see~\cite{MR0354618} and
\cite{MR1446124}).

Let $\alpha \in \Aut (C)$ defined over a minimal extension $\FF_{q^n}$ of $k =
\FF_q$ for some $n \geq 1$, and let $C_{\alpha}$ be the twist of $C$
corresponding to $\alpha$.  We construct the transformation $B$ from the
previous section by solving the equation $B^{\varphi} = A B$ for a suitable
matrix representation $A$ of $\alpha$. Since the curve is canonically embedded
in $\PP^2$, the representation of the action of $\Aut(C)$ on the regular
differentials gives a natural embedding of $\Aut(C)$ in $\GL_3(\FF_{q^n})$. We
let $A$ be the corresponding lift of $\alpha$ in this representation.
%
%
%
%First let $A_0$ be an arbitary lift of $\alpha$ to $\GL_3 (\overline{\FF}_q )$
%that is normalized in such a way that one of the entries of the matrix $A_0$
%equals $1$.
As $\Gal(\overline{\FF}_q / \FF_q)$ is topologically generated by $\varphi$ and
$\alpha$ is defined over a finite extension of $\FF_q$, there exists an integer
$m$ such that the cocycle relation $\alpha_{\sigma \tau} =
\alpha_{\tau}^{\sigma} \alpha_{\sigma}$ reduces to the equality
$A^{\varphi^{m-1}} \cdots A^{\varphi} A =\mathrm{Id}$. Using the multiplicative
form of Hilbert's Theorem 90, we let \[ B = P + \sum_{i=1}^{m-1} P^{\varphi^i}
A^{\varphi^{i-1}} \cdots A^{\varphi} A \] with $P$ a random matrix $3\times 3$
with coefficients in $\mathbb{F}_{q^m}$ chosen in such a way that at the end
$B$ is invertible. We will then have $B^{\varphi} = B A^{-1}$, the inverse of
the relation above, so that we can apply $B$ directly to the defining equation
of the quartic. Note that the probability of success of the algorithm is bigger
than ${1}/{4}$ (see \cite[Prop.1.3]{MR1446124}).

To estimate the complexity, we need to show that $m$ is not too large compared
with $n$. We have the following estimate.
\begin{lemma}
  Let $e$ be the exponent of $\Aut( C )$. Then $m\leq ne$.
\end{lemma}

\begin{proof}
  By definition of $n$ we have $\alpha^{\varphi^n} = \alpha$. Let $\gamma =
  \alpha^{\varphi^{n-1}} \cdots \alpha^{\varphi} \alpha$, and let $N$ be the
  order of $\gamma$ in $\Aut_{\mathbb{F}_{q^n}}( C )$. Since
  $\gamma^{\varphi^n} = \gamma$ and $\mathrm{Id} = \gamma^N =
  \alpha^{\varphi^{N n-1}} \cdots \alpha^{\varphi} \alpha$, we can take $m \leq
  n N \leq ne$.
\end{proof}

In practice we compute $m$ as the smallest integer such that
$\alpha^{\varphi^{m-1}} \cdots \alpha^{\varphi} \alpha$ is the identity.

\subsection{How to compute the twists by hand when $\#\Aut( C )$ is small}
\label{sec:how-compute-twists}

When the automorphism group is not too complicated, it is often possible to
obtain representatives of the classes in $H^1(\Gal ( \overline{\FF}_q/\FF_q ) ,
\Aut( C ))$ and then to compute the twists by hand, a method used in genus $2$
in \cite{cardona}. We did this for $\Aut(C)=\CG_2,\DG_4,\CG_3,\DG_8,\SG_3$.

Let us illustrate this in the case of $\DG_8$. As we have seen in Theorem
\ref{thm:smallstrata}, any curve $C/\FF_q$ with $\Aut(C) \simeq \DG_8$ is
$\overline{\FF}_q$-isomorphic with some curve $x^4 + x^2 y z + y^4 + a y^2 z^2
+ b z^4$ with $a,b\in \FF_q$. The problem splits up into several cases
according to congruences of $q-1 \pmod{4}$ and the class of $b \in
\FF_q^*/(\FF_q^*)^4$. We will assume that $4\,|\,(q-1)$ and $b$ is a fourth
power, say $b=r^4$ in $\FF_q$.  The $8$ automorphisms are then defined over
$\FF_q$: if $i$ is a square root of $-1$, the automorphism group is generated
by
\[
  S= \left[\begin{smallmatrix}
    1 & 0 & 0 \\
    0 & i & 0 \\
    0 & 0 & -i
  \end{smallmatrix}\right]
  \:\textrm{and}\:\:\:
  T = \left[
  \begin{smallmatrix}
    1 & 0 & 0 \\
    0 & 0 & r \\
    0 & r^{-1} & 0
  \end{smallmatrix}\right]
  \,.
\]
Representatives of the Frobenius conjugacy classes (which in this case reduce
to the usual conjugacy classes) are then $\mathrm{Id}$, $S$, $S^2$, $T$ and
$ST$. So there are $5$ twists.

Let us give details for the computation of the twist corresponding to the class
of $T$.  We are looking for a matrix $B$ such that $T B=B^{\varphi}$ up to
scalars. We choose $B$ such that $B (x,y,z)^t = (x, \alpha y + \beta z , \gamma
y + \delta z )^t$. Then we need to solve the following system:
\[
  \alpha^{\varphi} = r \gamma \,,
  \beta^{\varphi} = r \delta \,,
  \gamma^{\varphi} = r^{-1} \alpha \,,
  \delta^{\varphi} = r^{-1} \beta .
\]
The first equation already determines $\gamma$ in terms of $\alpha$. So we need
only satisfy the compatibility condition given by the second equation. Applying
$\varphi$, we get $\alpha^{\varphi^2}= (r \gamma)^{\varphi} = r
\gamma^{\varphi} = r (\alpha / r) = \alpha$. Reasoning similarly for $\beta$
and $\delta$, we see that it suffices to find $\alpha$ and $\beta$ in
$\mathbb{F}_{q^2}$ such that
\begin{math}
  \det
  \left(
    \begin{smallmatrix}
      \alpha & \beta \\
      \alpha^{\varphi} / r & \beta^{\varphi} / r
    \end{smallmatrix}
  \right)\neq 0.
\end{math}
We can take $\alpha = \sqrt{\tau}$ and $\beta =1$, with $\tau$ a primitive
element of $\FF_q^*$. Transforming, we get the twist
\[
  x^4 + r x^2 y^2 - r \tau x^2 z^2 + (a r^2 + 2 r^4) y^4 + (-2 a r^2 \tau + 12
  r^4 \tau) y^2 z^2 + (a r^2 \tau^2 + 2 r^4 \tau^2) z^4  = 0 .
\]

\section{Implementation and experiments}
\label{sec:impl-exper}

We combine the results obtained in Sections~\ref{sec:famquart}
and~\ref{sec:computation-twists} to compute a database of representatives of
$k$-isomorphism classes of genus $3$ non-hyperelliptic curves when $k = \FF_p$
is a prime field of small characteristic $p > 7$.

\subsection{The general strategy}
We proceed in two steps. The hardest one is to compute one representative
defined over $k$ for each $\bar{k}$-isomorphism class, keeping track of its
automorphism group. Once this is done, one can apply the techniques of
Section~\ref{sec:computation-twists} to get one representative for each
isomorphism class.

In order to work out the computation of representatives for the
$\bar{k}$-isomorphism classes, the naive approach would start by enumerating
all plane quartics over $k$ by using the 15 monomial coefficients $m_1$,
\ldots, $m_{15}$ ordered as in Equation~\eqref{eq:1} and for each new curve to
check whether it is smooth and not $\bar{k}$-isomorphic to the curves we
already kept as representatives. This would have  to be done for up to $p^{15}$
curves. For $p>29$, a better option is to use Proposition~\ref{prop:bergstrom}
to reduce to a family with $7$ parameters.

In both cases, checking for $\bar{k}$-isomorphism is relatively fast as we
make use of the so-called $13$ \emph{ Dixmier-Ohno invariants}. These are
generators for the algebra of invariants of ternary quartics forms under the
action of $\SL_3(\CC)$.  Among them 7 are denoted $I_3$, $I_6$, $I_9$,
$I_{12}$, $I_{15}$, $I_{18}$ and $I_{27}$ (of respective degree 3, 6, \ldots,
27 in the $m_i$'s) and are due to Dixmier~\cite{dixmier};  one also needs 6
additional invariants that are denoted $J_{9}$, $J_{12}$, $J_{15}$, $J_{18}$,
$I_{21}$ and $J_{21}$ (of respective degree 9, 12, \ldots, 21 in the $m_i$'s)
and that are due to Ohno \cite{ohno,giko}. These invariants behave well after
reduction to $\FF_p$ for $p>7$ and the discriminant $I_{27}$ is $0$ if and only
if the quartic is singular. Moreover, if two quartics have different
Dixmier-Ohno invariants (seen as points in the corresponding weighted
projective space, see for instance \cite{LR11}) then they are not
$\bar{k}$-isomorphic. We suspect that the converse is also true (as it is over
$\CC$). This is at least confirmed for our values of $p$ since at the end we
obtain $p^6+1$ $\overline{\FF}_p$-isomorphism classes, as predicted
by~\cite{bergstrom}.

The real drawback of this approach is that we cannot keep track of the
automorphism groups of the curves, which we need in order to compute the
twists. Unlike the hyperelliptic curves of genus $3$ \cite{LR11}, for which one
can read off the automorphism group from the invariants of the curve, we lack
such a dictionary for the larger strata of plane smooth quartics.

We therefore proceed by ascending up the strata, as summarized in
Algorithm~\ref{algo:enumerate}.  In light of Proposition~\ref{prop:desrep}, we
first determine the $\bar{k}$-isomorphism classes for quartics in the small
strata by using the representative families of Theorem~\ref{thm:smallstrata}.
In this case, the parametrizing is done in an optimal way and the automorphism
group is explicitly known. Once a stratum is enumerated, we consider a higher
one and keep a curve in this new stratum if and only if its Dixmier-Ohno
invariants have not already appeared. As mentioned at the end of
Section~\ref{sec:famquart}, this approach still finds all pointless curves
(except one for $\FF_{11}$) for $p \leq 29$. We can then use the generic
families in Proposition~\ref{prop:c2} and Proposition~\ref{prop:bergstrom}.

\begin{figure}[htbp]
  \begin{center}
    \parbox{0.9\linewidth}{%
      \begin{footnotesize}\SetAlFnt{\small\sf}%
        \begin{algorithm}[H]%
          \caption{Database of representatives for
            ${\FF}_p$-isomorphism classes of smooth plane quartics} %
          \label{algo:enumerate}%
          \SetKwInOut{Input}{Input} \SetKwInOut{Output}{Output} %
          \Input{A prime characteristic $p >7$.}  %
          \Output{A list ${\mathcal L}_p$ of mutually non-$\FF_p$-isomorphic
          quartics representing all isomorphism classes of smooth plane quartics
          over
          $\FF_p$.} \BlankLine %
          ${\mathcal L}_p := \emptyset$\;
          \For{$\mathbf{G} :=$\newline
            \hspace*{1cm}
            \begin{tabular}[b]{ll}
              %
              % Dimension 0
              $\mathbf{\Gg_{168}},
              \mathbf{\Gg_{96}},
              \mathbf{\Gg_{48}},
              \mathbf{\CG_9},$&\textsf{\small// Dim. 0 strata (first)}\\
              %
              % Dimension 1
              $\mathbf{\CG_6},
              \mathbf{\SG_4},
              \mathbf{\Gg_{16}},$&\textsf{\small// Dim. 1 strata (then)}\\
              %
              % Dimension 2
              $\mathbf{\SG_3},
              \mathbf{\CG_3},
              \mathbf{\DG_8},$&\textsf{\small// Dim. 2 strata (then)}\\
              %
              % Dimension 3
              $\mathbf{\DG_4},
              \mathbf{\CG_2},
              \mathbf{\{1\}}
              $&\textsf{\small// Dim. 3, 4 and 5 strata (finally)}\\
            \end{tabular}\newline
          }
          {
            \ForAll{quartics $Q$ defined by the families of\newline
              \hspace*{1cm}
              \begin{tabular}[b]{ll}
                Theorem~\ref{thm:smallstrata} & if $\mathbf{G}$ defines a
                stratum of
                dim. $\leq 3$,\\
                Proposition~\ref{prop:c2} & if $\mathbf{G} = \CG_2$,\\
                Proposition~\ref{prop:bergstrom} & if $\mathbf{G} = \{1\}$\\
              \end{tabular}\newline
            }
            {
              $(I_3:\ I_6:\ \ldots:\ J_{21}: I_{27}) :=$ Dixmier-Ohno
              invariants of $Q$\;
              \If{${\mathcal L}_p(I_3:\ I_6:\ \ldots:\ J_{21}: I_{27})$ is not
                defined }{
                ${\mathcal L}_p(I_3:\ I_6:\ \ldots:\ J_{21}: I_{27}) := \{Q$ and
                its twists$\}$\ \ \ \tcp{cf.~Section~\ref{sec:computation-twists}}
                \lIf{${\mathcal L}_p$ contains $p^6+1$
                entries}{\KwRet{${\mathcal L}_p$}}
              }
            }
          }
        \end{algorithm}
      \end{footnotesize}
    }
  \end{center}
\end{figure}

\subsection{Implementation details}

We split our implementation of Algorithm~\ref{algo:enumerate} into two parts.
The first one, developed with the \textsc{Magma} computer algebra software,
handles quartics in the strata of dimension 0, 1, 2 and 3. These strata have
many fewer points than the ones with geometric automorphism group $\CG_2$ and
$\{1\}$ but need linear algebra routines to compute twists. The second part has
been developed in the \textsc{C}-language for two reasons: to efficiently
compute the Dixmier-Ohno invariants in the corresponding strata and to decrease
the memory needed. We now discuss these two issues.

\subsubsection{Data structures.}
\label{sec:memory-constraints}

We decided to encode elements of $\FF_p$ in bytes. This limits us to $p < 256$,
but this is not a real constraint since larger $p$ seem as yet infeasible (even
considering the storage issue).  As most of the time is spent computing
Dixmier-Ohno invariants, we group the multiplications and additions that occur
in these calculations as much as possible in 64-bit microprocessor words before
reducing modulo $p$. This decreases the number of divisions as much as
possible.

To deal with storage issues in Step 6 of Algorithm~\ref{algo:enumerate}, only
the 13 Dixmier-Ohno invariants of the quartics are made fully accessible in
memory; we store the full entries in a compressed file. These entries are
sorted by these invariants and additionally list the automorphism group, the
number of twists, and for each twist, the coefficients of a representative
quartic, its automorphism group and its number of points.

\subsubsection{Size of the hash table.}
\label{sec:size-hash-table}

We make use of an open addressing hash table to store the list ${\mathcal L}_p$
from Algorithm~\ref{algo:enumerate}. This hash table indexes $p^5$ buckets, all
of equal size $(1+\varepsilon) \times p $ for some overhead $\varepsilon$.
Given a Dixmier-Ohno 13-tuple of invariants, its first five elements
(eventually modified by a bijective linear combination of the others to get a
more uniform distribution) give us the address of one bucket of the table of
invariants. We then store the last eight elements of the Dixmier-Ohno 13-tuple
at the first free slot in this bucket. The total size of the table is thus
$8\,(1+\varepsilon) \times p^6$ bytes.

All the buckets do not contain the same number of invariants at the end of the
enumeration, and we need to fix $\varepsilon$ such that it is very unlikely
that one bucket in the hash table goes over its allocated room.
To this end, we assume that Dixmier-Ohno invariants behave like random
13-tuples, \textit{i.e.} each of them has probability $1/p^5$ to address a
bucket. Experimentally, this assumption seems to be true. Therefore the
probability that one bucket $\mathcal B$ contains $n$ invariants after $k$
trials follows a binomial distribution,
\begin{displaymath}
  \Prob({\mathcal B} = n) = \binom{n}{k}\times \frac{(p^5-1)^{k-n}}{(p^5)^k} =
  \binom{n}{k} \times \left(\frac{1}{p^5}\right)^n\times \left(1-\frac{1}{p^5}\right)^{k-n} \,.
\end{displaymath}
\begin{wrapfigure}{r}{0.5\textwidth}
  %\vspace*{-0.5cm}
    \resizebox{0.5\textwidth}{0.15\textheight}{
      \includegraphics{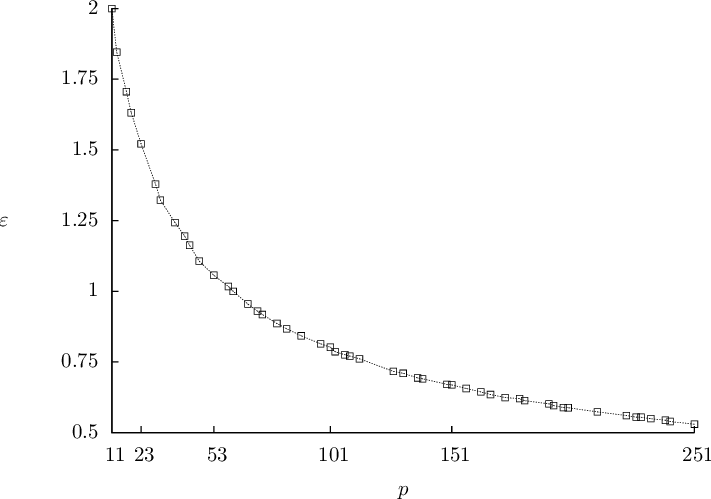}
    }
    \vspace*{-0.5cm}
    \caption{Overhead $\varepsilon$}
    \label{fig:epsilon}
    \vspace*{-0.5cm}
\end{wrapfigure}
Now let $k \approx p^6$. Then $k\times (1/p^5) \approx p$, which is a fixed
small parameter. In this setting, Poisson approximation yields $\Prob({\mathcal
B} = n) \simeq p^n\, e^{-p} / n\,!$, so the average number of buckets that
contain $n$ entries at the end is about $p^5\,\Prob({\mathcal B} = n) \simeq
p^{5+n}\,e^{-p} / n\,!$ and it remains to choose $n = (1+\varepsilon)\,p$, and
thus $\varepsilon$, such that this probability is negligible. We draw
$\varepsilon$ as a function of $p$ when this probability is smaller than
$10^{-3}$ in Figure~\ref{fig:epsilon}. For $p=53$, this yields a hash table of
340 gigabytes.

\subsection{Results and first observations}
\label{sec:results}

We have used our implementation of Algorithm~\ref{algo:enumerate} to compute
the list ${\mathcal L}_p$ for primes $p$ between 11 and 53. Table~\ref{tab:Lp}
gives the corresponding timings and database sizes (once stored in a compressed
file).
Because of their size, only the databases $\mathcal{L}_p$ for $p = 11$ or $p =
13$, and a program to use them, are available online\footnote{
  \url{http://perso.univ-rennes1.fr/christophe.ritzenthaler/programme/qdbstats-v3_0.tgz}.
}.

\begin{table}[htbp]
  \renewcommand{\arraystretch}{1.3}\addtolength{\tabcolsep}{-3pt}
  \centering
  \begin{scriptsize}
    \begin{tabular}{l|cccccccccccc}
      $p$ & 11 & 13 & 17 & 19 & 23 & 29 & 31 & 37 & 41 & 43 & 47 & 53 \\
      \hline
      \hline
      Time & 42s & 1m 48s & 10m  & 20m 30s& 1h 7m  & 4h 36m & 6h 48m & 22h 48m &
      1d 23h & 2d 7h & 5d 22h & 7d 19 \\
      Db size & 27Mb & 68Mb & 377Mb & 748Mb & 2.5Gb & 11.5 Gb & 16Gb & 51Gb &
      97Gb & 128Gb & 224Gb & 460Gb\\
      \hline
    \end{tabular}
  \end{scriptsize}\medskip
  \caption{Calculation of ${\mathcal L}_p$ on a 32 AMD-Opteron 6272 based server}
  \label{tab:Lp}
\end{table}

As a first use of our database, and as a sanity check, we can try to
interpolate formulas for the number of $\FF_p$- or
$\overline{\FF}_p$-isomorphism classes of genus 3 plane quartics over $\FF_p$
with given automorphism group. The resulting polynomials in $p$ are given in
Table~\ref{tab:auto}. The `$+[a]_{\ \mathrm{condition}}$' notation means that
$a$ should be added if the `condition' holds.

\begin{table}[htbp]
  \renewcommand{\arraystretch}{1.5}\addtolength{\tabcolsep}{-3pt}
  \centering
  \begin{scriptsize}
  \begin{tabular}{l|ll}
      $G$ & \#$\overline{\FF}_p$-isomorphism classes &
      \#${\FF}_p$-isomorphism classes \\\hline\hline
      %
      % Dimension 0
      $\mathbf{\Gg_{168}}$& $\mathbf{1}$&$\mathbf{4+[2]_{p = 1,2,4 \bmod 7}}$\\
      $\mathbf{\Gg_{96}}$& $\mathbf{1}$&$\mathbf{6+[4]_{4|p+3}}$\\
      $\Gg_{48}$& $\mathbf{1}$& $\mathbf{4+[10]_{12 | p+11}+[2]_{12|p +7}+[4]_{12|p+5}}$ \\
      $\mathbf{\CG_9}$ & $\mathbf{1}$ & $\mathbf{1+[8]_{9|p+8}+[2]_{9|p+5}+[6]_{9|p+2}}$\\\hline
      %
      % Dimension 1
      $\mathbf{\CG_6}$& $\mathbf{p-2}$ & $\mathbf{2\times  (1+[2]_{3|p+2}) \times \#\overline{\FF}_p}$\textbf{-iso.}   \\
      $\mathbf{\SG_4}$& $\mathbf{p-4-[2]_{p = 1,2,4 \bmod 7}}$ &$\mathbf{5\times \#\overline{\FF}_p}$\textbf{-iso}.\\
      $\Gg_{16}$& $\mathbf{p-2}$ & $\mathbf{2 \times (2\,(p-3) + [p-2]_{4|p+3})}$\\\hline
      %
      % Dimension 2
      $\mathbf{\SG_3}$&$\mathbf{p^2-3\,p+4+[2]_{p = 1,2,4 \bmod 7}}$&$\mathbf{3\times \#\overline{\FF}_p}$\textbf{-iso.}\\
      $\mathbf{\CG_3}$&$\mathbf{p^2-p}$&$\mathbf{(1+[2]_{3|p+2})\times \#\overline{\FF}_p}$\textbf{-iso.}\\
      $\mathbf{\DG_8}$&$\mathbf{p^2-4\,p+6+[2]_{p = 1,2,4 \bmod 7}}$&$\mathbf{4\times\#\overline{\FF}_p}$\textbf{-iso}.$\mathbf{-3\,p+8}$\\\hline
      %
      % Dimension 3
      $\mathbf{\DG_4}$&$\mathbf{{p}^{3}-3\,{p}^{2}+5\,p-5}$&$\mathbf{2\,{p}^{3}-8\,{p}^{2}+17\,p-19}$\\\hline
      %
      % Dimension 4
      $\mathbf{\CG_2}$&${p}^{4}-2\,{p}^{3}+2\,{p}^{2}-3\,p+1-[2]_{p = 1,2,4 \bmod 7}$&$\mathbf{2\times \#\overline{\FF}_p}$\textbf{-iso.}\\\hline
      %
      % Generic case
      $\mathbf{\{1\}}$&${p}^{6}-{p}^{4}+{p}^{3}-2\,{p}^{2}+3\,p-1$&$\mathbf{\#\overline{\FF}_p}$\textbf{-iso.}\\\hline\hline
      Total &$\mathbf{p^6+1}$&$
      \begin{array}[t]{r}
        {p}^{6}+{p}^{4}-{p}^{3}+2\,{p}^{2}-4\,p-1+2\,(p \bmod 4)\hfill\\
        +2\,[{p}^{2}+p+2-(p \bmod 4)]_{p = 1,4,7 \bmod 9}
         +[6]_{p = 1  \bmod 9} \\
        +[2\,p+6]_{p=1 \bmod 4}
        +[2]_{p = 1,2,4 \bmod 7}\\

      \end{array}
      $\\\hline
    \end{tabular}
  \end{scriptsize}\medskip
  \caption{Number of isomorphism classes of plane quartics
  with given automorphism group}%($p>7$)}
  \label{tab:auto}
\end{table}

Most of these formulas can actually be proved (we \textbf{emphasize} the ones
we are able to prove in Table~\ref{tab:auto}). In particular, it is possible to
derive the number of most of the \#$\overline{\FF}_p$-isomorphic classes from
the representative families given in Theorem~\ref{thm:smallstrata}; one merely
needs to consider the degeneration conditions between the strata. For example,
for the strata of dimension $1$, the singularities at the boundaries of the
strata of dimension $1$ corresponding to strata with larger automorphism group
are given by $\FF_p$-points, except for the stratum $\SG_4$. The latter stratum
corresponds to singular curves for $a \in \left\{ -2,-1,2 \right\}$, and the
Klein quartic corresponds to $a = 0$. But the Fermat quartic corresponds to
both roots of the equation $a^2 + 3 a + 18$ (note that the family for the
stratum $\SG_4$ is no longer representative at that boundary point).  The
number of roots of this equation in $\FF_p$ depends on the congruence class of
$p$ modulo $7$.

One proceeds similarly for the other strata of small dimension; the above
degeneration turns out to be the only one that gives a dependence on $p$. To
our knowledge, the point counts for the strata $\CG_2$ and $\{1\}$ are still
unproved. Note that the total number of $\overline{\FF}_p$-isomorphism classes
is known to be $p^6+1$ by~\cite{bergstrom}, so the number of points on one
determines the one on the other.

Determining the number of twists is a much more cumbersome task, but can still
be done by hand  by making explicit the cohomology classes of
Section~\ref{sec:computation-twists}.  For the automorphism groups
$\mathbf{\Gg_{168}}$, $\mathbf{\Gg_{96}}$, $\Gg_{48}$ and $\mathbf{\SG_4}$,  we
have recovered the results published by Meagher and Top in~\cite{MR2678623} (a
small subset of the curves defined over $\FF_p$ with automorphism group
$\Gg_{16}$ was studied there as well).

\subsection{Distribution according to the number of points}
\label{subsec:dist}

Once the lists $\mathcal{L}_p$ are determined, the most obvious invariant
function on this set of isomorphism classes is the number of rational points of
a representative of the class. To observe the distributions of these classes
according to their number of points was the main motivation of our extensive
computation.  In Appendix~\ref{sec:num}, we give some graphical interpretations
of the results for prime fields $\FF_p$ with $11 \leq p \leq 53$\footnote{The
  numerical values we used for these graphs can be found
  at~\url{http://perso.univ-rennes1.fr/christophe.ritzenthaler/programme/qdbstats-v3_0.tgz}.
}.

Although we are still at an early stage of exploiting the data, we can make the
following remarks:
\begin{enumerate}
  \item Among the curves whose number of points is maximal or minimal, there
    are only curves with non-trivial automorphism group, except for a pointless
    curve over $\FF_{11}$ mentioned at the end of Section~\ref{sec:big}. While
    this phenomenon is not true in general (see for
    instance~\cite[Tab.2]{Rit-serre} using the form $43,\# 1$ over
    $\FF_{167}$), it shows that the usual recipe to construct maximal curves,
    namely by looking in families with large non-trivial automorphism groups,
    makes sense over small finite fields. It also shows that to observe the
    behavior of our distribution at the borders of the Hasse-Weil interval, we
    have to deal with curves with many automorphisms, which justifies the
    exhaustive search we made.
  \item Defining the trace $t$ of a curve $C/\FF_q$ by the usual formula
    $t=q+1-\#C(\FF_q)$, one sees in Fig.~\ref{fig:1} that the ``normalized
    trace'' $\tau=t/ \sqrt{q} $ accurately follows the asymptotic distribution
    predicted by the general theory of Katz-Sarnak~\cite{katz-sarnak}. For
    instance, the theory predicts that the mean normalized trace should
    converge to zero when $q$ tends to infinity. We found the following
    estimates for $q=11,17,23,29,37,53$: $$4 \cdot 10^{-3}, \quad 1 \cdot
    10^{-3}, \quad 4 \cdot 10^{-4}, \quad 2 \cdot 10^{-4}, \quad 6 \cdot
    10^{-5}, \quad 3 \cdot 10^{-5}.$$
  \item Our extensive computations enable us to spot possible fluctuations with
    respect to the symmetry of the limit distribution of the trace, a
    phenomenon that to our knowledge has not been encountered before (see
    Fig.~\ref{fig:2}). These fluctuations are related to the \emph{Serre's
    obstruction} for genus $3$~\cite{Rit-serre} and do not appear for genus
    $\leq 2$ curves. Indeed, for these curves (and more generally for
    hyperelliptic curves of any genus), the existence of a quadratic twist
    makes the distribution completely symmetric.  The fluctuations also cannot
    be predicted by the general theory of Katz and Sarnak, since this theory
    depends only on the monodromy group, which is the same for curves,
    hyperelliptic curves or abelian varieties of a given genus or dimension.
    Trying to understand this new phenomenon is a challenging task and indeed
    the initial purpose of constructing our database.
  %  As a final observation, looking at the graphs, we see that when $t>0$ is
  %  larger than $0.3 \cdot g \lfloor 2\sqrt{q} \rfloor$, the number of curves
  %  with $q+1-t$ points is at most the number of curves with $q+1+t$. In other
  %  words, there are more curves with many points than curves with few points.
  %  This remains true when counting curves by the `mass formula', \ie
  %  weighting a curve by its  $\FF_p$-automorphism group.  This may not be
  %  surprising when $p<31$, as in this case $p+1-3 \lfloor 2 \sqrt{p}
  %  \rfloor<0$ and the non-negative part of the Hasse-Weil interval is not
  %  symmetric itself.  But we wonder if this phenomenon is still valid for any
  %  $q$, as our data for larger primes does seem to imply. It is worth
  %  noticing that a positive answer could be used to give a closed formula for
  %  the maximal number of points of a genus $3$ curve over $\FF_q$, which is a
  %  30-year old open and fascinating problem due to Serre~\cite{serre-point}.
\end{enumerate}

\bibliographystyle{abbrv}

\appendix

\section{Generators and normalizers}
\label{sec:gener-norm}

As mentioned in Remark~\ref{rem:conj}, the automorphism groups in
Theorem~\ref{thm:aut} have the property that their isomorphism class determines
their conjugacy class in $\PGL_3 (K)$. Accordingly, the families of curves in
Theorem~\ref{thm:aut} have been chosen in such a way that they allow a common
automorphism group as subgroup of $\PGL_3 (K)$. We proceed to describe the
generators and normalizers of these subgroups; these can be found by direct
computation or by using~\cite[Lemma 2.3.8]{huggins-thesis}. The generators that
we give below of the automorphism groups themselves are in fact lifts to $\GL_3
(K)$ obtained by considering the tangent representation in terms of the basis
of differentials corresponding to the monomials $x, y, z$.

In what follows, we consider $\GL_2 (K)$ as a subgroup of $\PGL_3 (K)$ via the
map $A \mapsto \left[ \begin{smallmatrix} 1 & 0 \\ 0 & A \end{smallmatrix}
\right]$. The group $D (K)$ is the group of diagonal matrices in $\PGL_3 (K)$,
and $T (K)$ is its subgroup consisting of those matrices in $D (K)$ that are
non-trivial only in the upper left corner. We consider $\SG_3$ as a subgroup
$\widetilde{\SG}_3$ of $\PGL_3 (K)$ by the permutation action that it induces
on the coordinate functions, and we denote by $\widetilde{\SG}_4$ the degree
$2$ lift of $\SG_4$ to $\PGL_3 (K)$ generated by the matrices\medskip
\begin{footnotesize}
  \[
  \begin{bmatrix}
    1 & 0 & 0 \\
    0 & \zeta_8 & 0 \\
    0 & 0 & \zeta_8^{-1}
  \end{bmatrix}\,,
  \ \ \ \frac{-1}{i + 1}
  \begin{bmatrix}
    1 & 0 & 0 \\
    0 & \zeta_4 & -\zeta_4 \\
    0 & 1 & 1
  \end{bmatrix}
  \,.
  \]
\end{footnotesize}\bigskip

\begin{theorem}\label{thm:norm}
  The following are generators for the automorphism groups $G$ in
  Theorem~\ref{thm:aut}, along with the isomorphism classes and generators of
  their normalizers $N$ in $\PGL_3 (K)$.
  \begin{enumerate}
    \item $\{ 1 \}$ is generated by the unit element. $N = \PGL_3 (K)$.
    \item $\CG_2 = \langle \alpha \rangle$, where $\alpha (x, y, z) = (x, -y,
      -z)$. $N = \GL_2 (K)$.
    \item $\DG_4 = \langle \alpha, \beta \rangle$, where $\alpha (x, y, z) =
      (x, -y, -z)$ and $\beta (x, y, z) = (-x, y, -z)$. $N = D (K) \rtimes
      \widetilde{\SG}_3$.
    \item $\CG_3 = \langle \alpha \rangle$, where $\alpha (x, y, z) =
      (\zeta_3^{-1} x, \zeta_3 y, \zeta_3 z)$. $N = \GL_2 (K)$.
    \item $\DG_8 = \langle \alpha, \beta \rangle$, where $\alpha (x, y, z) =
      (x, \zeta_4 y, \zeta_4^{-1} z)$ and $\beta (x, y, z) = (-x, -z, -y)$. $N
      = T(K) \widetilde{\SG}_4$.
    \item $\SG_3 = \langle \alpha, \beta \rangle$, where $\alpha (x, y, z) =
      (x, \zeta_3 y, \zeta_3^{-1} z)$ and $\beta (x, y, z) = (-x, -z, -y)$. $N
      = T(K) \widetilde{\SG_3}$.
    \item $\CG_6 = \langle \alpha \rangle$, where $\alpha (x, y, z) =
      (-\zeta_3^{-1} x, \zeta_3 y, -\zeta_3 z)$. $N = D (K)$.
    \item $\Gg_{16} = \langle \alpha, \beta, \gamma \rangle$, where $\alpha (x,
      y, z) = (-x, \zeta_4 y, \zeta_4 z)$,  $\beta (x, y, z) = (-x, y, -z)$,
      and $\gamma (x, y, z) = (-x, -z, -y)$. $N = T(K) \widetilde{\SG}_4$.
    \item $\SG_4 = \langle \alpha, \beta, \gamma, \delta \rangle$, where
      $\alpha (x, y, z) = (x, -y, -z)$,  $\beta (x, y, z) = (-x, y, -z)$,
      $\gamma (x, y, z) = (y, z, x)$, and $\delta (x, y, z) = (-y, -x, -z)$.
      $N = G$.
    \item $\CG_9 = \langle \alpha \rangle$, where $\alpha (x, y, z) = (x,
      \zeta_9^2 y, \zeta_9^{-4} z)$. $N = D(K)$.
    \item $\Gg_{48} = \langle \alpha, \beta, \gamma \rangle$, where $\alpha
      (x, y, z) = (-x, \zeta_4 y, \zeta_4 z)$, $\beta (x, y, z) = (\zeta_3 x,
      \zeta_3^{-1} y, \zeta_3 z)$,  and $\gamma (x, y, z) = -(1 / \sqrt{3})
      (\sqrt{3} x, y + 2 z, y - z)$. $N$ is $\PGL_3 (K)$-conjugate to $N = T(K)
      \widetilde{\SG}_4$.
    \item $\Gg_{96} = \langle \alpha, \beta, \gamma, \delta \rangle$, where
      $\alpha (x, y, z) = (-x, \zeta_4 y, \zeta_4 z)$, $\beta (x, y, z) =
      (\zeta_4 x, -y, \zeta_4 z)$, $\gamma (x, y, z) = (y, z, x)$,  and $\delta
      (x, y, z) = (-y, -x, -z)$. $N = G$.
    \item $\Gg_{168} = \langle \alpha, \beta, \gamma \rangle$, where $\alpha
      (x, y, z) = (\zeta_7 x, \zeta_7^4 y, \zeta_7^2 z)$,  $\beta(x, y, z) =
      (y, z, x)$,  and
      \begin{align*}
        \gamma (x, y, z) = \frac{-1}{\sqrt{-7}} (
          & (\zeta_7^4 - \zeta_7^3) x + (\zeta_7 - \zeta_7^6) y + (\zeta_7^2 - \zeta_7^5) z, \\
          & (\zeta_7 - \zeta_7^6) x + (\zeta_7^2 - \zeta_7^5) y + (\zeta_7^4 - \zeta_7^3) z, \\
          & (\zeta_7^2 - \zeta_7^5) x + (\zeta_7^4 - \zeta_7^3) y + (\zeta_7 - \zeta_7^6) z  ) ,
      \end{align*}
      where $\sqrt{-7}$ is the Gauss sum $\sum_{i=1}^6 \left(\frac{7}{i}\right)
      \zeta_7^i$. $N = G$.
  \end{enumerate}
\end{theorem}

\section{Numerical results}
\label{sec:num}

Given a prime number $p$, we let $N_{p,3}(t)$ denote the number of
$\FF_p$-isomorphism classes of non-hyperelliptic curves of genus $3$ over
$\FF_p$ whose trace equals $t$. %Let $m=\lfloor 2 \sqrt{p} \rfloor$ and
Define
$$N^{\textrm{KS}}_{p,3}(\tau) = \frac{\sqrt{p}}{\# \Mm_3(\FF_p)} \cdot N_{p,3}(t), \quad t=\lfloor \sqrt{p} \cdot \tau \rfloor, \quad  \tau \in [-6,6],$$
which is the normalization of the distribution of the trace as
in~\cite{katz-sarnak}. Our numerical results are summarized in
Fig.~\ref{fig:3}.

\begin{figure}[htbp]\label{fig:KSref}
  \centering
  \subcaptionbox%
  {Graph of $N^{\textrm{KS}}_{p,3}(\tau)$ \label{fig:1}}%
  {\includegraphics[height=6cm]{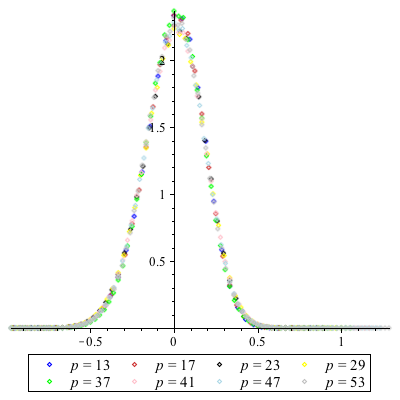}}
  \subcaptionbox%
  {Graph of $N^{\textrm{KS}}_{p,3}(\tau)-N^{\textrm{KS}}_{p,3}(-\tau)$ \label{fig:2}}%
  {\includegraphics[height=6cm]{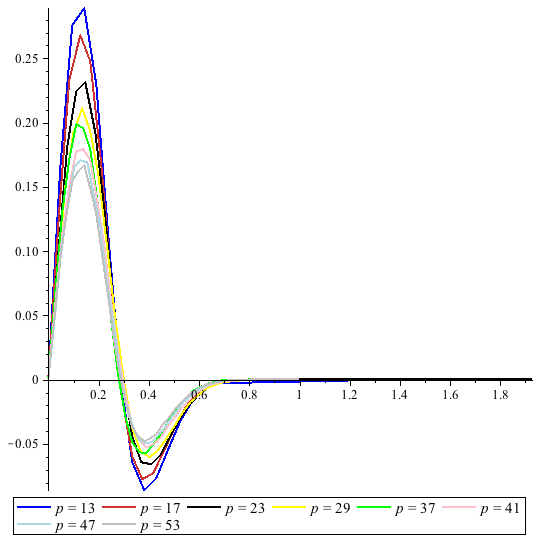}}
  \caption{Trace distribution \label{fig:3}}
  \end{figure}
%
%
%
%\begin{figure}[htbp]
% \centering
%  %
%
%  \caption{Graph of $N^{\textrm{KS}}_{p,3}(\tau)$ \label{fig:1}}%
%  \includegraphics[height=6cm]{distributionNKS3.png}
%  \end{figure}
%  %
%
%\begin{figure}[htbp]
% \centering
%  %
%  \caption{Graph of $N^{\textrm{KS}}_{p,3}(\tau)-N^{\textrm{KS}}_{p,3}(-\tau)$ \label{fig:2}}%
%  \includegraphics[height=6cm]{distributionVKS3.png}
%  %
%%  \caption{Trace distribution \label{fig:3}}
%  \end{figure}

\end{document}